\documentclass[12pt]{article}
\usepackage{amsmath,amsthm,amssymb,latexsym,color}
\usepackage{hyperref}

\theoremstyle{plain}
\newtheorem{theor}{Theorem}[section]
\newtheorem*{theorA}{Theorem~A}

\newtheorem{lemma}[theor]{Lemma}

\theoremstyle{remark}

\def\R{{\mathbb R}}

\def\Z{{\mathbb Z}}
\def\Event{{\mathcal E}}
\def\Proj{{\rm Proj}}
\def\Prob{{\mathbb P}}
\def\Exp{{\mathbb E}}
\def\dist{{\rm dist}}
\def\spn{{\rm span}}
\def\col{{\rm col}}
\def\supp{{\rm supp\,}}

\def\crosspol{{\mathcal K}}
\def\conv{{\rm conv\,}}
\def\Net{{\mathcal N}}
\def\Gluskin{{\mathcal G}}
\def\mtrx{{\mathcal A}}
\def\Fone{{\mathcal F_1}}
\def\Ftwo{{\mathcal F_2}}
\def\Fcomp{{\mathcal F}}
\def\BM{{\rm d_{BM}}}
\def\Mink{{\mathcal M_n}}
\def\idmat{{\rm Id_n}}

\title{On the Banach--Mazur distance to cross-polytope}
\author{Konstantin Tikhomirov\footnote{Princeton University, NJ; email: kt12@math.princeton.edu.
The research is partially supported by the Simons Foundation.}}

\begin{document}

\maketitle

\begin{abstract}
Let $n\geq 3$, and let $B_1^n$ be the standard $n$--dimensional cross-polytope (i.e.\ the convex hull of standard coordinate vectors
and their negatives). We show that there exists a symmetric convex body $\Gluskin_m$ in $\R^n$ such that
the Banach--Mazur distance $\BM(B_1^n,\Gluskin_m)$ satisfies $\BM(B_1^n,\Gluskin_m)\geq n^{5/9}\log^{-C}n$,
where $C>0$ is a universal constant.
The body $\Gluskin_m$ is obtained as a typical realization of a random polytope in $\R^n$ with
$2m:=2n^C$ vertices (for a large constant $C$). The result improves upon an earlier estimate of S.~Szarek which gives
$\BM(B_1^n,\Gluskin_m)\geq c n^{1/2}\log n$ (with a different choice of $m$).
This shows in a strong sense that the cross-polytope
(or the cube $[-1,1]^n$) cannot be an ``approximate'' center of the Minkowski compactum.
\end{abstract}

\section{Introduction}

The Minkowski (or the Banach--Mazur) compactum $\Mink$ is defined as the collection of all origin-symmetric $n$--dimensional
convex bodies equipped with the distance function
$$\BM(K,L):=\inf\big\{d\geq 1:\;\exists T\in {\rm GL_n(\R)}\mbox{ such that }K\subset T(L)\subset dK\big\}$$
(in this note, we do not consider non-symmetric bodies).
The classical theorem of F.~John \cite{John} asserts that $\BM(K,B_2^n)\leq\sqrt{n}$
for all $K\in\Mink$, where $B_2^n$ is the standard Euclidean ball.
The question of estimating $\sup\limits_{K\in\Mink}\BM(K,B_\infty^n)$
(or, equivalently, $\sup\limits_{K\in\Mink}\BM(K,B_1^n)$) has attracted considerable attention of researchers.
Here and in what follows, by $B_p^n$ ($1\leq p\leq \infty$) we denote the unit ball of space $\ell_p^n$;
in particular, $B_\infty^n=[-1,1]^n$ and $B_1^n=\conv\{\pm e_1,\dots,\pm e_n\}$,
where $e_1,e_2,\dots,e_n$ is the standard vector basis in $\R^n$.

Currently best upper bound for the quantity $\sup\limits_{K\in\Mink}\BM(K,B_1^n)$
is $Cn^{5/6}$ due to A.~Giannopoulos \cite{Gian}
(we refer to earlier results of J.~Bourgain--S.~Szarek \cite{BoSz} and
S.~Szarek--M.~Talagrand \cite{SzTa} giving upper estimates $n\exp(-c\sqrt{\log n})$ and $Cn^{7/8}$,
respectively,
as well as to a result of P.~Youssef \cite{Youssef} for estimate with improved constant $2n^{5/6}$).
The connection of the problem
with the property of {\it restricted invertibility} of matrices and the proportional Dvoretzky--Rogers factorization
has been intensively explored in literature; we refer, among others, to papers
\cite{BoSz, SzTa, Gian, Youssef, FrYoussef, NaorYoussef, NaorYoussef2}.

For the lower bound, the only available result up to this writing, due to S.~Szarek,
asserts that $\sup\limits_{K\in\Mink}\BM(K,B_1^n)\geq c\sqrt{n}\log n$ for a universal constant $c>0$
\cite{Szarek}.
The proof of the lower estimate in \cite{Szarek} involves two crucial ingredients.
First is construction of
a family of random polytopes which, together with an $\varepsilon$--net
argument and some probabilistic relations, reduces
the problem to studying relative positions of independent Gaussian vectors in $\R^n$.
In a different form, such construction was first used in context of
geometric functional analysis by E.~Gluskin \cite{Gluskin}
and allowed him to solve a crucial problem of estimating the diameter of the Minkowski compactum
(we refer to survey \cite{MT} for more information; see also \cite{LRTJ} for more recent applications of Gluskin's construction).
The second element of Szarek's proof is an estimate of singular values of the standard Gaussian matrix
which can be viewed as a non-asymptotic analog of the Marchenko--Pastur law \cite{MP} for the spectrum
of sample covariance matrices, with very strong probability bounds.

The result of S.~Szarek shows that the cross-polytope (or the cube) cannot be a ``center'' of the Minkowski compactum
similar to the Euclidean ball. However, the result leaves open a possibility that
the distance of any convex body to the cube is bounded above by the
square root of dimension times a polylogarithmic multiple.
A conjecture of A.~Naor (personal communication)
that the bound obtained in \cite{Szarek} is suboptimal, is confirmed in the main result of this note
(see also \cite[remark~2]{SzTa} for related discussion).
\begin{theorA}
Let $n\geq 3$ and let $m:=n^3$. Further, let $\Gluskin_m$ be the random symmetric polytope constructed as
the convex hull of $2m$ vectors
$\pm G_1,\pm G_2,\dots,\pm G_m$ where $G_1,G_2,\dots,G_m$ are independent standard Gaussian ($N(0,\idmat)$)
vectors in $\R^n$.
Then with positive probability $\BM(\Gluskin_m,B_1^n)\geq n^{5/9}\log^{-C}n$.
Here, $C>0$ is a universal constant.
\end{theorA}
The starting point of our proof is the same as in the S.~Szarek's work \cite{Szarek}: we construct Gluskin's random polytope
$\Gluskin_m$
to estimate its Banach--Mazur distance to an $n$--dimensional cross-polytope.
However, instead of working with singular values we develop a combination of geometric
and probabilistic arguments to ``directly'' estimate the Gaussian measure
of cross-polytopes inscribed into a given realization of Gluskin's polytope.

The structure of the paper is the following: in Section~\ref{s: setup}
we recall those elements of Szarek's construction \cite{Szarek} that are also used in the present paper.
In Section~\ref{s: dd} we give a high-level overview of our strategy,
which comprises two essential parts:
constructing a special event of probability close to one which catches the geometric properties useful for us
and estimating the Gaussian measure of cross-polytopes from a special class.
The event is constructed in Section~\ref{s: event}, while the Gaussian measure of cross-polytopes
is computed in Section~\ref{s: crosspol}.
Finally, in Section~\ref{s: completion} we choose parameters and complete the proof.

\section{Preliminaries}\label{s: setup}

Let us start with basic notation. Given a finite set $I$, we denote its cardinality by $|I|$.
The convex hull of a set of points $S$ in a linear space is denoted by $\conv(S)$.
Given a matrix $A$, its columns are denoted by $\col_1(A),\col_2(A),\dots$
Given a vector $v$, let $\supp v$ be its support.
For two convex bodies $K$ and $L$, $K+L$ denotes their Minkowski sum.
Further, for any set $K$ in $\R^n$ and an $m\times n$ matrix $A$, by $A(K)$ we denote
the linear image of $K$ in $\R^m$.
Let $(\Omega,\Prob)$ be a probability space.
For any random variable/vector/matrix $\xi$ on $\Omega$, by $\xi(\omega)$ we denote
the realization of $\xi$ at a point $\omega\in\Omega$.

We introduce several global parameters:
\begin{equation*}
\mbox{A large integer $n$; a number $m\geq n^2$; $\varepsilon\in(0,1/2]$ and $\rho\geq 1$.}
\end{equation*}
In what follows, $m$ will be responsible for the geometry of our random polytope
(we will take $2m$ to be the number of generating vectors in Gluskin's polytope in $\R^n$);
$\varepsilon$ will be used to define a net (discretization) on a set of linear operators;
$\rho$ will serve as a lower bound for the Banach--Mazur distance of the Gluskin polytope to a cross-polytope.

We will impose additional assumptions on the parameters in various statements below;
the parameters will be explicitly chosen at the end of the note
at the optimization stage.

Let $G_i$, $i\leq m$ be jointly independent standard Gaussian vectors in $\R^n$,
and $\Gamma$ be the $n\times m$ standard Gaussian matrix with columns $G_i$.
We define the random symmetric convex body $\Gluskin_m$ as the convex hull of $\pm G_i$, $i\leq m$.
These random convex polytopes, introduced to high-dimensional convex geometry by E.~Gluskin \cite{Gluskin}
(the original definition was slightly different),
turned out extremely useful.

\subsection{Discretization}

All observations in this subsection, up to minor modifications, repeat those from \cite{Szarek}.

\medskip

To show that there is a symmetric convex set with the Banach--Mazur distance to $B_1^n$
at least $\rho$, it is sufficient to show that
\begin{equation}\label{eq: basis}
\Prob\big\{\mbox{There is a symmetric cross-polytope $P$ with $\rho\,\Gluskin_m\supset\rho P\supset \Gluskin_m$}\big\}<1.
\end{equation}
By Caratheodory's theorem, for a given point $\omega\in\Omega$
of the probability space, any cross-polytope $P$ inscribed into the realization $\Gluskin_m(\omega)$,
can be represented as $P=\Gamma(\omega) A(B_1^n)$
for an $m\times n$ matrix $A$ (determined by $\omega$) such that the cardinality of support of every column of $A$ is at most $n$
and the $\ell_1^n$--norm of each column at most one.
Define the class $\mtrx_{m,n}$ of matrices $A$ in $\R^{m\times n}$ satisfying the conditions
$$|\supp\col_i(A)|\leq n;\quad\quad \|\col_i(A)\|_1\leq 1;\quad\quad i\leq n.$$

With this definition, \eqref{eq: basis} would follow as long as we show that
the event
$$\Event_{*}:=\big\{\mbox{There is $A\in\mtrx_{m,n}$ such that $\rho\, \Gamma A(B_1^n)\supset \Gluskin_m$}\big\}$$
has probability strictly less than one.

Further, define a discretization of $\mtrx_{m,n}$ as follows. Let $\Net\subset\mtrx_{m,n}$
be the set of all $m\times n$ matrices $A=(a_{ij})$ in $\mtrx_{m,n}$ such that
$\varepsilon^{-1}a_{ij}\in\Z$ for all indices $i,j$.
It is easy to see that for any $A=(a_{ij})\in\mtrx_{m,n}$ there is $A'=(a_{ij}')\in\Net$ such that $\max\limits_{i,j}|a_{ij}-a_{ij}'|< \varepsilon$.

We have the following elementary lemma:
\begin{lemma}\label{l: basic discretization}
Let parameters $m,n,\rho,\varepsilon$ satisfy the additional assumptions
$m\leq n^{10}$ and $\varepsilon\rho n^2\leq 1$.
Denote by $\Event_{\text{\tiny\ref{l: basic discretization}}}$ the event
$$\Event_{\text{\tiny\ref{l: basic discretization}}}
:=\big\{\mbox{There is $A\in\Net$ such that $2\rho\, \Gamma A(B_1^n)\supset \Gluskin_m$}\big\}.$$
Then $\Prob(\Event_{*})\leq \Prob(\Event_{\text{\tiny\ref{l: basic discretization}}})+2^{-n}$.
\end{lemma}
\begin{proof}
Assume that the difference $\Event_{*}\setminus \Event_{\text{\tiny\ref{l: basic discretization}}}$
is non-empty, and fix any $\omega\in\Event_{*}\setminus \Event_{\text{\tiny\ref{l: basic discretization}}}$.
Let $r\geq 0$ be the largest real number such that $r B_2^n\subset \Gluskin_m(\omega)$,
and let $A=A(\omega)=(a_{ij})$ be a matrix in $\mtrx_{m,n}$ such that
$\rho\, \Gamma (\omega)A(\omega)(B_1^n)\supset \Gluskin_m(\omega)$.
Take $A'=(a_{ij}')\in\Net$ with $\max\limits_{i,j}|a_{ij}-a_{ij}'|< \varepsilon$ and $\supp\col_i(A')\subset\supp\col_i(A)$, $i\leq n$.
It is not difficult to see that
$\Gamma (\omega)A'(B_1^n)+\varepsilon n\max\limits_{j\leq m}\|G_j(\omega)\|_2\,B_2^n\supset \Gamma (\omega)A(B_1^n)$.
Indeed, take any $x\in \Gamma (\omega)A(B_1^n)$; then $x=\sum_{i=1}^n \alpha_i \Gamma(\omega)\col_i(A)$ for some
$(\alpha_1,\dots,\alpha_n)\in B_1^n$. Set $y:=\sum_{i=1}^n \alpha_i \Gamma(\omega)\col_i(A')$
and observe that $\|\Gamma(\omega)\col_i(A'-A)\|_2\leq \varepsilon n\max\limits_{j\leq m}\|G_j(\omega)\|_2$, $i\leq n$.
Hence, $\|x-y\|_2\leq \varepsilon n\max\limits_{j\leq m}\|G_j(\omega)\|_2$, so
$x\in \Gamma (\omega)A'(B_1^n)+\varepsilon n\max\limits_{j\leq m}\|G_j(\omega)\|_2\,B_2^n$.

The above inclusion, together with the definition of $r$, gives
$$\Gamma (\omega)A'(B_1^n)+\varepsilon n\max\limits_{j\leq m}\|G_j(\omega)\|_2\,\frac{\rho}{r}\, \Gamma (\omega)A(B_1^n)
\supset \Gamma (\omega)A(B_1^n).$$
This implies that $\varepsilon n\max\limits_{j\leq m}\|G_j(\omega)\|_2\,\frac{\rho}{r}> \frac{1}{2}$,
since otherwise we would have $\Gamma (\omega)A'(B_1^n)\supset \frac{1}{2}\Gamma (\omega)A(B_1^n)$,
which is impossible as $\omega\notin \Event_{\text{\tiny\ref{l: basic discretization}}}$.
Next, observe that the definition of $r$ implies that there is $u\in S^{n-1}$ such that
$|\langle u,G_i(\omega)\rangle|\leq r$ for all $i\leq m$ whence $s_{\min}(\Gamma(\omega)^T )\leq r\sqrt{m}$.

Summarizing, we showed that the difference $\Event_{*}\setminus \Event_{\text{\tiny\ref{l: basic discretization}}}$
is contained in the event
$$\Big\{
\varepsilon n\max\limits_{j\leq m}\|G_i\|_2\,\frac{\rho\sqrt{m}}{s_{\min}(\Gamma^T )}> \frac{1}{2}
\Big\}.$$
Standard, by now, concentration properties for $\|G_i\|_2$ and $s_{\min}(\Gamma^T)$
(see, for example, \cite[Theorem~II.6 and Theorem~II.13]{DS}), together with our assumptions on parameters,
imply the result.
\end{proof}

\subsection{Distances to random linear spans}

The next lemma is an elementary application of standard concentration results for Gaussian
variables. We provide the proof for Reader's convenience.
Recall that $\Gamma$ is the ``global object'' of the proof defined as the $n\times m$ standard Gaussian matrix
with columns $G_i$, $i\leq m$.

\begin{lemma}[Distances to linear spans]\label{l: dist to spans I}
There are universal constants $C_{\text{\tiny\ref{l: dist to spans I}}},c_{\text{\tiny\ref{l: dist to spans I}}}>0$
with the following property.
Assume that $n\geq C_{\text{\tiny\ref{l: dist to spans I}}}$, $n/2\leq u\leq n$, $1\leq k\leq u/2$, $\tau\geq C_{\text{\tiny\ref{l: dist to spans I}}}$,
$\delta\in(1/k,1]$;
and fix any $m\times u$ non-random matrix $B$
of full rank $u$
such that each column has Euclidean norm at most one.
Denote $H_i:=\Gamma (\col_i(B))$, $i\leq u$.
Further, for any permutation $\sigma$ of $[u]$ let $\Event_\sigma$ be the event that
$$\big|\big\{i:\,u-k+1\leq i\leq u,\;\dist(H_{\sigma(i)},\spn\{H_{\sigma(j)},\;j\leq i-1\})\leq \tau\sqrt{n-u+k}\big\}\big|\geq (1-\delta)k.$$
Then $\Prob(\Event_\sigma)\geq 1-e^{-c_{\text{\tiny\ref{l: dist to spans I}}}\tau^2 \delta(n-u+k)k}$ and, moreover,
$\Prob(\bigcap\limits_{\sigma\in\Pi_u}\Event_\sigma)\geq 1-u^k\,e^{-c_{\text{\tiny\ref{l: dist to spans I}}}\tau^2 \delta(n-u+k)k}$,
where $\Pi_u$ is the set of all permutations on $[u]$.
\end{lemma}
\begin{proof}
Since the linear span of $\{H_{\sigma(j)},\;j\leq u-k\}$ is completely determined by $\Gamma$
and the set $\{\sigma(j),\,u\geq j\geq u-k+1\}$,
it is enough to show the first assertion of the lemma: that for any fixed permutation $\sigma$ of $[u]$ the probability of the event
$$\big|\big\{i:\,u-k+1\leq i\leq u,\;\dist(H_{\sigma(i)},\spn\{H_{\sigma(j)},\;j\leq i-1\})\leq \tau\sqrt{n-u+k}\big\}\big|\geq (1-\delta)k$$
is bounded from below by $1-e^{-c\tau^2 \delta(n-u+k)k}$,
for an appropriate universal constant $c>0$.

Fix a permutation $\sigma$ and any subset $I\subset \{u-k+1,\dots,u\}$ of cardinality $\lceil \delta k\rceil$.
Basic properties of the Gaussian distribution (specifically, rotational invariance), together with the definition of $H_i$'s, immediately imply that
there exists a decomposition
$$H_{\sigma(i)}=\widetilde H_{\sigma(i)}+\hat H_{\sigma(i)},\quad\quad i\leq u,$$
such that for every $i>1$, $\widetilde H_{\sigma(i)}$ is a linear combination of $H_{\sigma(j)}$'s, $j<i$
(i.e.\ belongs to their linear span), while $\hat H_{\sigma(i)}$
is a non-zero multiple of the standard Gaussian vector in $\R^n$
independent from $\{H_{\sigma(j)},\,j<i\}$, with $\Exp\|\hat H_{\sigma(i)}\|_2^2\leq n$.
In particular, conditioned on any $(i-1)$--dimensional realization of $\spn\{H_{\sigma(j)},\,j<i\}$, the distance
$$\dist(\hat H_{\sigma(i)},\spn\{H_{\sigma(j)},\;j< i\})=\dist(H_{\sigma(i)},\spn\{H_{\sigma(j)},\;j< i\})$$
is equidistributed with the Euclidean norm
of a multiple of the standard Gaussian vector in $\R^{n-i+1}$ (with the multiplication coefficient not greater than one).
Thus, applying a standard concentration inequality for Lipschitz functions in the Gauss space (see, for example,
\cite[Theorem~II.6]{DS}), we get
$$\Prob\big\{
\dist(H_{\sigma(i)},\spn\{H_{\sigma(j)},\;j< i\})> \tau\sqrt{n-u+k}\mbox{ for all }i\in I\big\}
\leq e^{-c'\tau^2(n-u+k)|I|}
$$
for a suitable universal constant $c'>0$.
Taking the union bound over all possible choices of $I$, we get that the event
$$\big|\big\{i:\;u-k+1\leq i\leq u,\;\dist(H_{\sigma(i)},\spn\{H_{\sigma(j)},\;j<i\})\leq \tau\sqrt{n-u+k}\big\}\big|\geq (1-\delta)k$$
has probability at least $1-\delta^{-\delta k} e^{-c'\tau^2\delta(n-u+k)k}$.
The result follows.
\end{proof}

\section{Decomposition and structuring}\label{s: dd}

In this section we develop a way to estimate from above probability of the event
$\Event_{\text{\tiny\ref{l: basic discretization}}}$
from the previous section. The approach is significantly different from the one in paper \cite{Szarek}.
One of main challenges of the net-based approach is to control probabilities in a way that admits some
sort of a union bound. 
If the cardinality of the net $\Net$ introduced in the previous section were very small, we would be able
to bound the probability of $\Event_{\text{\tiny\ref{l: basic discretization}}}$ simply by summing up
probability estimates of 
inclusion $2\rho\, \Gamma A(B_1^n)\supset \Gluskin_m$ for each $A\in\Net$.
However, the size of $\Net$ is greater than $2^{n^2}$ and this approach would require to take
rather small value for $\rho$ to make sure the summation produces a number less than one.

To deal with this issue, we will partition every matrix $A$ from $\Net$ into
a matrix with ``very sparse'' columns and a matrix with columns of small Euclidean norms.
We introduce another global parameter $\alpha\in(0,1/2]$ whose value will be determined at the optimization stage.
Let $x$ be a vector in $\R^m$.
We say that $x$ is of {\it type $(\alpha+)$} if $\|x\|_1\leq 1$, and
$|\supp x|\leq 1/\alpha$.
Further, $x$ is of {\it type $(\alpha-)$} if $|\supp x|\leq n$, $\|x\|_1\leq 1$ and $\|x\|_2< \sqrt{\alpha}$.
Roughly speaking, type $(\alpha+)$ corresponds to very sparse vectors while $(\alpha-)$
consists of moderately sparse vectors of small Euclidean norm. It is easy to see that any vector $y\in\R^m$
with $\|y\|_1\leq 1$ and with cardinality of support at most $n$
can be decomposed into the sum $y_1+y_2$, where $y_1$ is of type $(\alpha+)$, $y_2$ is of type $(\alpha-)$,
and $y_1,y_2$ have disjoint supports.

Define two mappings $\Fone,\Ftwo:\Net\to\Net$ as follows:
given $A=(a_{ij})\in\Net$, let $\Fone(A)$ be the $m\times n$ matrix with entries $a_{ij}{\bf 1}_{|a_{ij}|\geq\alpha}$,
where ${\bf 1}_{|a_{ij}|\geq\alpha}$ is the indicator of the boolean expression ``$|a_{ij}|\geq\alpha$''.
Further, we set $\Ftwo(A)$ to be the $m\times n$ matrix with entries $a_{ij}{\bf 1}_{|a_{ij}|<\alpha}$.
Finally, set $\Fcomp(A)$ to be the $m\times 2n$ matrix obtained by concatenating $\Fone(A)$ and $\Ftwo(A)$.
Obviously, $A(B_1^n)\subset \Fcomp(A)(2B_1^{2n})$ for any $A\in\Net$.
This elementary relation turns out extremely useful in our context.
It shows that every point of the random cross-polytope $\Gamma A(B_1^n)$
is a convex combination of random vectors of two types: vectors which are $\alpha^{-1}$--sparse linear combinations
of $G_i$'s and vectors which have relatively small expected Euclidean norms. The number
of vectors of the first type is relatively small (because of the $\alpha^{-1}$--sparsity of corresponding linear combinations)
allowing an efficient net-argument. At the same time, vectors of the second type are ``short'',
which enables us to control their influence even though the number of corresponding linear combinations
is relatiely large.

In the following lemma we formulate sufficient conditions which allow us to bound the Banach--Mazur distance
between the Gluskin polytope $\Gluskin_m$ and an $n$-dimensional cross-polytope.
\begin{lemma}[Decomposition]\label{l: decomposition}
Let $\Event$ be an event of non-zero probability having the following structure:
$\Event=\bigcap\limits_{A\in\Net}\Event_A$, where for every $A\in\Net$ the event $\Event_A$
is measurable with respect to $\sigma$-algebra generated by vectors $G_j$, with $j\in\bigcup\limits_{i\leq n}\supp\,\col_i(A)$.
Let $\widetilde G$ be the standard Gaussian vector in $\R^n$ independent
from $\Gamma $.
Then
\begin{align*}
\Prob(\Event_{\text{\tiny\ref{l: basic discretization}}}\cap\Event)\leq
|\Net|\max\limits_{A\in\Net}\sup\limits_{\omega\in\Event_A}
\Prob\big(\big\{&\omega'\in\Omega:\;\mbox{For some $I\subset[2n]$ with $|I|=n$ we have}\\
&\mbox{$\widetilde G (\omega')\in 4\rho\,\big(\conv\big\{\pm\col_i(\Gamma (\omega)\Fcomp(A)),\;\;i\in I\big\}\big)$}\;\big\}\big)^{m-n^2}.
\end{align*}
\end{lemma}
\begin{proof}
Pick any point $\omega\in \Event_{\text{\tiny\ref{l: basic discretization}}}\cap\Event$
(we assume that the intersection is non-empty as otherwise there is nothing to prove).
In view of the definition of $\Event_{\text{\tiny\ref{l: basic discretization}}}$, there is $A=A(\omega)\in\Net$
such that for any $j\leq m$ we have $G_j(\omega)\in 2\rho\,\Gamma (\omega)A(B_1^n)
\subset 4\rho\,\Gamma (\omega)\Fcomp(A)(B_1^{2n})$.
By Caratheodory's theorem, this implies that there is a subset $I=I(\omega,j)\subset[2n]$ of cardinality $n$ such that
$$G_j(\omega)\in 4\rho\,\big(\conv\big\{\pm\col_i(\Gamma (\omega)\Fcomp(A)),\;\;i\in I\big\}\big).$$
In what follows, we are only interested in indices $j\in [m]\setminus \bigcup\limits_{i\leq n}\supp\,\col_i(A)$,
which will enable us to use independence. In view of the above, we get
\begin{align*}
\Prob(\Event_{\text{\tiny\ref{l: basic discretization}}}\cap\Event)
\leq \Prob\big(\big\{\exists A\in\Net \;\;\;&\mbox{such that for any $j\notin\bigcup\limits_{i\leq n}\supp\,\col_i(A)$}\\
&\mbox{and some
$I=I(j)\subset[2n]$ with $|I|=n$ we have}\\
&\mbox{$G_j\in 4\rho\,\big(\conv\big\{\pm\col_i(\Gamma \Fcomp(A)),\;\;i\in I\big\}\big)$}\;\big\}\cap \Event\big)\\
&\hspace{-3.1cm}\leq |\Net|\,
\max\limits_{A\in\Net}
\Prob\big(\big\{\mbox{For any $j\notin\bigcup\limits_{i\leq n}\supp\,\col_i(A)$ and some}\\
&\mbox{$I=I(j)\subset[2n]$ with $|I|=n$ we have}\\
&\mbox{$G_j\in 4\rho\,\big(\conv\big\{\pm\col_i(\Gamma \Fcomp(A)),\;\;i\in I\big\}\big)$}\;\big\}\cap \Event_A\big).
\end{align*}
Observe that, by the conditions on $\Event_A$, vectors $G_j$, $j\notin\bigcup\limits_{i\leq n}\supp\,\col_i(A)$,
are independent from $\Event_A$ and $\Gamma \Fcomp(A)$.
Thus, the last expression can be estimated from above by
\begin{align*}
|\Net|\max\limits_{A\in\Net}\sup\limits_{\omega\in\Event_A}
\Prob\big(\big\{\omega'\in\Omega:\;&\mbox{For some $I\subset[2n]$ with $|I|=n$ we have}\\
&\mbox{$\widetilde G(\omega')\in 4\rho\,\big(\conv\big\{\pm\col_i(\Gamma (\omega)\Fcomp(A)),\;\;i\in I\big\}\big)$}\;\big\}\big)^{m-n^2}.
\end{align*}

\end{proof}

\bigskip

The last lemma provides a very useful mechanism of bounding probability $\Prob(\Event_{\text{\tiny\ref{l: basic discretization}}})$
from above. Indeed, if the events $\Event_A$ are such that for any $\omega\in\Event_A$ we have
\begin{align}
\Prob\big(\big\{\omega'\in\Omega:\;&\mbox{For some $I\subset[2n]$ with $|I|=n$ we have}\nonumber\\
&\mbox{$\widetilde G(\omega')\in 4\rho\,\big(\conv\big\{\pm\col_i(\Gamma (\omega)\Fcomp(A)),\;\;i\in I\big\}\big)$}\;\big\}\big)\leq 1/2,
\label{eq: eventa cond}
\end{align}
then, by choosing $m$ sufficiently large (so that the power $m-n^2$ is big enough)
we will beat the cardinality of $\Net$ and obtain an upper estimate for $\Prob(\Event_{\text{\tiny\ref{l: basic discretization}}}\cap \Event)$.
If, at the same time, the probability of $\Event$ is close to one, this will imply upper bounds for $\Prob(\Event_{\text{\tiny\ref{l: basic discretization}}})$.
Thus, the rest of the argument consists of two major steps:
\begin{itemize}

\item Find an appropriate event $\Event_\cap=\bigcap\limits_{A\in\Net}\Event_A$ satisfying
conditions of Lemma~\ref{l: decomposition};

\item Show that for any $A\in\Net$ and $\omega\in\Event_A$ we have \eqref{eq: eventa cond}.

\end{itemize}

\section{Constructing event $\Event_\cap$}\label{s: event}

We define the following parametric family of cross-polytopes in $\R^n$.
Given $1\leq k\leq n$ and $h>0$, define
\begin{align*}
&\mbox{$\crosspol(k,h)$ --- collection of all origin-symmetric cross-polytopes $T\subset\R^n$}\\
&\mbox{having the following structure:
$T=\conv\{\pm x_1,\pm x_2,\dots,\pm x_n\}$, where}\\
&\mbox{for each permutation $\sigma$ of $[n]$ we have:}\\
&\big|\big\{i:\;n-k+1\leq i\leq n,\;\dist(x_{\sigma(i)},\spn\{x_{\sigma(j)},\;j<i\})\leq h\big\}\big|\geq k/4.
\end{align*}


Further, let $s,\widetilde s,\tau,\delta$ be parameters satisfying $n\geq s\geq \widetilde s\geq 1$ and $\tau\geq C_{\text{\tiny\ref{l: dist to spans I}}}$.
For every matrix $A\in\Net$ define two events $\Event_A^1(s,\widetilde s)$ and
$\Event_A^2(\widetilde s,\delta,\tau)$ as follows:
\begin{align*}
\Event_A^1(s,\widetilde s):=\big\{&\mbox{For every $I_1\subset[n]$, $I_2\subset\{n+1,\dots,2n\}$
with $|I_1\cup I_2|=n$, $|I_1|\geq n-\widetilde s$,}\\
&\mbox{such that vectors $\col_i(\Gamma \Fcomp(A))$, $i\in I_1\cup I_2$, are linearly independent, we}\\
&\mbox{have $\conv\{\pm \col_i(\Gamma \Fcomp(A)),\;i\in I_1\cup I_2\}\in \crosspol(s,C_{\text{\tiny\ref{l: dist to spans I}}}\sqrt{2s})$}\big\}
\end{align*}
and
\begin{align*}
\Event_A^2(\widetilde s,\delta,\tau):=\big\{&\mbox{For every $I_1\subset[n]$, $I_2\subset\{n+1,\dots,2n\}$
with $|I_1\cup I_2|=n$, $|I_2|> \widetilde s$, such}\\
&\mbox{that vectors $\col_i(\Gamma \Fcomp(A))$, $i\in I_1\cup I_2$, are linearly independent, we have}\\
&\mbox{$\big|\big\{i\in I_2:\;\dist(\col_i(\Gamma \Fcomp(A)),\spn\{\col_j(\Gamma \Fcomp(A)),\;j\in (I_1\cup I_2)\cap[i-1]\})$}\\
&\leq \tau\sqrt{\alpha|I_2|}\big\}\big|\geq (1-\delta)|I_2|\big\}.
\end{align*}
Set $\Event_A:=\Event_A^1\cap\Event_A^2$ and
$\Event_\cap:=\bigcap\limits_{A\in\Net}\Event_A$ (for brevity, we sometimes suppress the list of parameters).
It is not difficult to see that by the definition of $\Event_A$ (regardless of the values of the parameters),
the event is measurable with respect to the $\sigma$-algebra generated by
$G_j$, with $j\in\bigcup\limits_{i\leq n}\supp\,\col_i(A)$. Thus, the intersection $\Event_\cap$
satisfies conditions of Lemma~\ref{l: decomposition}.
In the next two lemmas, which conclude this section, we will show that, with an appropriate choice of parameters,
the event $\Event_\cap$ has probability close to one.

\begin{lemma}\label{l: parameters I}
There is a universal constant $C_{\text{\tiny\ref{l: parameters I}}}>0$ with the following property.
Assume that
$m/\varepsilon\leq n^{10}$,
$1>\delta\geq (\log n)^{-1}$, $n\geq \widetilde s\geq \log^2 n$ and, additionally, assume that
$\tau\geq C_{\text{\tiny\ref{l: dist to spans I}}}$ is such that
$$\min\bigg(\frac{\tau^2\delta \widetilde s^2\alpha}{n},\frac{\tau^2\delta \widetilde s}{n}\bigg)\geq C_{\text{\tiny\ref{l: parameters I}}}\log n.$$
Then
$$\Prob\big(\bigcap\limits_{A\in \Net}\Event_A^2(\widetilde s,\delta,\tau)\big)\geq 1-\frac{1}{n}.$$
\end{lemma}
\begin{proof}
Fix any $p\in\{0,1,\dots,n\}$ and consider the collection $T_p$ of all $m\times n$ matrices $B$
satisfying the following condition: there is $A\in\Net$ such that $B$ is a submatrix of $\Fcomp(A)$,
where the first $p$ columns of $B$ are from $\Fone(A)$ and the last $n-p$ columns are from $\Ftwo(A)$.
Obviously, the set $T_p$ is contained within the collection $T_p'$ of all $m\times n$ matrices where the first
$p$ columns are of type $(\alpha+)$, the last $n-p$ columns are of type $(\alpha-)$, and, additionally, all
entries of the matrix are from $\varepsilon\Z$.
It is not difficult to see that the cardinality of $T_p'$ (hence, $T_p$) can be bounded from above as
$$|T_p'|\leq\bigg({m\choose {\lfloor\alpha^{-1}\rfloor}}(\varepsilon/3)^{-1/\alpha}\bigg)^p
\bigg({m\choose n}(\varepsilon/3)^{-n}\bigg)^{n-p}.
$$
We need to show that probability of the event $\bigcap\limits_{A\in \Net}\Event_A^2$,
that is, the event
\begin{align*}
\big\{&\mbox{For any $p\in\{0,1,\dots,n-\widetilde s-1\}$ and any $B\in T_{p}$ of full rank, we have}\\
&\mbox{$\big|\big\{p+1\leq i\leq n:\;\dist(\col_i(\Gamma B),\spn\{\col_j(\Gamma B),\;
j<i\})$}\leq \tau\sqrt{\alpha (n-p)}\big\}\big|\\
&\geq (1-\delta)(n-p)\big\},
\end{align*}
is close to one.
Take any $p\in\{0,1,\dots,n-\widetilde s-1\}$.
By the definition of $T_p$, we have $\|\col_i(B)\|_2\leq\sqrt{\alpha}$ for all $B\in T_p$ and $i>p$.
Hence, in view of Lemma~\ref{l: dist to spans I}, for any $B\in T_p$ of full rank we have
\begin{align*}
\Prob\big\{
&\big|\big\{p+1\leq i\leq n:\;\dist(\col_i(\Gamma B),\spn\{\col_j(\Gamma B),\;
j<i\})\\
&\leq \tau\sqrt{\alpha(n-p)}\big\}\big|\geq (1-\delta)(n-p)\big\}
\geq 1-e^{-c_{\text{\tiny\ref{l: dist to spans I}}}\tau^2 \delta (n-p)^2}.
\end{align*}
This immediately implies
\begin{align*}
\Prob\big(\bigcap\limits_{A\in \Net}\Event_A^2\big)&\geq 1-\sum\limits_{p=0}^{n-\widetilde s-1}
\bigg({m\choose {\lfloor\alpha^{-1}\rfloor}}(\varepsilon/3)^{-1/\alpha}\bigg)^p
\bigg({m\choose n}(\varepsilon/3)^{-n}\bigg)^{n-p}
\,e^{-c_{\text{\tiny\ref{l: dist to spans I}}}\tau^2 \delta (n-p)^2}\\
&=:1-\sum\limits_{p=0}^{n-\widetilde s-1}J_p.
\end{align*}
A direct computation shows that for every $p\in\{0,1,\dots,n-\widetilde s-1\}$, we have
\begin{align*}
J_p\leq \bigg(\frac{3e\alpha m}{\varepsilon}\bigg)^{\alpha^{-1}p}
\bigg(\frac{3em}{\varepsilon n}\bigg)^{n(n-p)}
\,e^{-c_{\text{\tiny\ref{l: dist to spans I}}}\tau^2 \delta (n-p)^2}.
\end{align*}
Note that, by the assumptions on the parameters assuming the constant $C_{\text{\tiny\ref{l: parameters I}}}$ is sufficiently
large), we have
$$c_{\text{\tiny\ref{l: dist to spans I}}}\tau^2 \delta (n-p)^2\geq
2\Big(\alpha^{-1}p\log\frac{3e\alpha m}{\varepsilon}+n(n-p)\log\frac{3em}{\varepsilon n}\Big),$$
whence
$J_p\leq e^{-c_{\text{\tiny\ref{l: dist to spans I}}}\tau^2 \delta (n-p)^2/2}$.
Summing over all admissible $p$, we get the result.
\end{proof}

\begin{lemma}\label{l: parameters II}
There is a universal constant $C_{\text{\tiny\ref{l: parameters II}}}$ with the following property. Assume that
$m/\varepsilon\leq n^{10}$, $n\geq s\geq 4\widetilde s\geq 4\log^2 n$, and that
$\frac{\widetilde s^2\alpha}{n}\geq C_{\text{\tiny\ref{l: parameters II}}}\log n$.
Then
$$\Prob\big(\bigcap\limits_{A\in \Net}\Event_A^1(s,\widetilde s)\big)\geq 1-\frac{1}{n}.$$
\end{lemma}
\begin{proof}
Let $T_p$ be defined the same way as in the proof of Lemma~\ref{l: parameters I}.
For any $p=n-\widetilde s,n-\widetilde s+1,\dots,n$, denote by
$\Proj_p:\R^n\to\R^n$ the orthogonal projection onto first $p$ standard coordinate vectors
and observe that the set of matrices
$$Q_p:=\big\{B\Proj_p:\;B\in T_p\big\}$$
has cardinality
$$
|Q_p|\leq\bigg({m\choose {\lfloor\alpha^{-1}\rfloor}}(\varepsilon/3)^{-1/\alpha}\bigg)^p
\leq \bigg(\frac{3e\alpha m}{\varepsilon}\bigg)^{\alpha^{-1}p}.
$$
Further, observe that the intersection $\bigcap\limits_{A\in \Net}\Event_A^1(s,\widetilde s)$
coincides with the event $\bigcap\limits_{p=n-\widetilde s}^n\Event_p$, where
\begin{align*}
\Event_p:=\big\{&\mbox{For every $A\in\Net$, $I_1\subset[n]$, $I_2\subset\{n+1,\dots,2n\}$ with $|I_1\cup I_2|=n$, $|I_1|= p$,}\\
&\mbox{such that vectors $\col_i(\Gamma\Fcomp(A))$, $i\in I_1\cup I_2$, are linearly independent, we}\\
&\mbox{have $\conv\{\pm \col_i(\Gamma\Fcomp(A)),\;i\in I_1\cup I_2\}\in \crosspol(s,C_{\text{\tiny\ref{l: dist to spans I}}}\sqrt{2s})$}\big\}.
\end{align*}
In turn, probability of each event $\Event_p$ can be bounded from below by probability of the event
\begin{align*}
\Event_p':=\Big\{
&\mbox{For any $B\in Q_p$ of rank $p$ and any permutation $\sigma$ of $[p]$ we have}\\
&\big|\big\{i:\;p-s+1\leq i\leq p,\;\dist(\col_{\sigma(i)}(\Gamma B),\\
&\spn\{\col_{\sigma(j)}(\Gamma B),\;j\leq i-1\})\leq C_{\text{\tiny\ref{l: dist to spans I}}}\sqrt{2s}\big\}\big|\geq s/2\Big\}.
\end{align*}
Let us prove the last assertion.
Take any $\omega\in\Event_p'\setminus\Omega_0$, where $\Omega_0$ is the event (of probability zero)
that for some $A\in\Net$, $\Gamma\Fcomp(A)$ contains an $n\times n$ submatrix of deficient rank.
Further, let $\sigma$ be any permutation of $[n]$;
let $A\in\Net$ and let $I_1\subset[n]$, $I_2\subset\{n+1,\dots,2n\}$ with $|I_1\cup I_2|=n$, $|I_1|= p$.
Denote by $B$ the $m\times n$ matrix with first $p$ columns coincident with $\col_i(\Fcomp(A))$, $i\in I_1$
(with ordering of columns preserved),
and last $n-p$ zero columns. Clearly, $B\in Q_p$. Define permutation $\sigma'$ of $[p]$ is such a way that
$\sigma^{-1}(\sigma'(\ell))$ is increasing with $\ell$ on $[p]$.
By the definition of event $\Event_p'$, we have that there are at least $s/2$ indices
$i\in\{p-s+1,\dots,p\}$ such that
$$\dist(\col_{\sigma'(i)}(\Gamma(\omega) B),
\spn\{\col_{\sigma'(j)}(\Gamma(\omega) B),\;j\leq i-1\})\leq C_{\text{\tiny\ref{l: dist to spans I}}}\sqrt{2s}.$$
Let $R:I_1\cup I_2\to[n]$ be the order-preserving bijection and set
$x_i:=\col_{R^{-1}(i)}(\Gamma(\omega)\Fcomp(A))$, $i\leq n$.
Then the last condition can be rewritten as
\begin{equation}\label{eq: aux 1}
\dist(x_{\sigma'(i')},
\spn\{x_{\sigma'(j)},\;j\leq i'-1\})\leq C_{\text{\tiny\ref{l: dist to spans I}}}\sqrt{2s}\;\;\mbox{for at least $s/2$ ind.\ $i'\geq p-s+1$}.
\end{equation}
The conditions on $p$ and $\widetilde s$ imply that
the set $S:=\{\sigma(i):\;n-s+1\leq i\leq n\}\cap [p]$ has cardinality at least $s-\widetilde s\geq 3s/4$.
Further, for any $i$ with $\sigma(i)\in S$ we have $\sigma(i)=\sigma'(i')$, where,
by the definition of $\sigma'$, necessarily $i'\geq p-(n-i)\geq p-s+1$, and, moreover,
$\spn\{x_{\sigma'(j)},\;j\leq i'-1\}\subset \spn\{x_{\sigma(j)},\;j\leq i-1\}$.
In particular, we have
$$\dist(x_{\sigma(i)},
\spn\{x_{\sigma(j)},\;j\leq i-1\})\leq \dist(x_{\sigma'(i')},
\spn\{x_{\sigma'(j)},\;j\leq i'-1\}).$$
This, together with \eqref{eq: aux 1}, implies
that
$$\dist(x_{\sigma(i)},
\spn\{x_{\sigma(j)},\;j\leq i-1\})\leq C_{\text{\tiny\ref{l: dist to spans I}}}\sqrt{2s}\;\;\mbox{for at least $s/4$ indices $i\geq n-s+1$},$$
whence $\omega\in \Event_p$. Thus, indeed $\Prob(\Event_p)\geq\Prob(\Event_p')$.

\medskip

Applying Lemma~\ref{l: dist to spans I} (with $\delta:=1/2$, $u:=p$, $k:=s$, $\tau:=C_{\text{\tiny\ref{l: dist to spans I}}}$)
and taking the union over all $B\in Q_p$, we get
$$\Prob(\Event_p')\geq 1-
n^s\,e^{-c_{\text{\tiny\ref{l: dist to spans I}}}C_{\text{\tiny\ref{l: dist to spans I}}}^2 s^2/2}\bigg(\frac{3e\alpha m}{\varepsilon}\bigg)^{\alpha^{-1}n}
\geq 1-e^{-c_{\text{\tiny\ref{l: dist to spans I}}}C_{\text{\tiny\ref{l: dist to spans I}}}^2 s^2/4},
$$
where the last relation follows by the assumptions on parameters.
Taking the union bound over admissible $p$, we get the result.
\end{proof}

\section{Gaussian measure of tilted cross-polytopes}\label{s: crosspol}

Let us recall our proof strategy as outlined in Section~\ref{s: dd}.
We have constructed the event $\Event_\cap=\bigcap\limits_{A\in\Net} \Event_A$ and
essentially showed in Lemmas~\ref{l: parameters I}
and~\ref{l: parameters II} that under an appropriate choice of parameters $\Event_\cap$
has probability close to one.
As the second step of the proof, we will show that (again, with appropriately chosen parameters)
each $\omega\in\Event_A=\Event_A^1\cap\Event_A^2$ satisfies \eqref{eq: eventa cond}.
Clearly, \eqref{eq: eventa cond} can be interpreted as a statement about the Gaussian measure of a union of
cross-polytopes of the form $4\rho\,\conv\big\{\pm \col_i(\Gamma(\omega)\Fcomp(A)),\;i\in I\big\}$,
where $I$ is a subset of $[2n]$ of cardinality $n$.
We will estimate the measure of each such cross-polytope in one of the two ways depending on the cardinality of the set $I\cap[n]$.
When $I\cap[n]$ is ``large'' (that is, vast majority generating vectors of the cross-polytope are realized
as $\alpha^{-1}$--sparse combinations of columns of $\Gamma(\omega)$, assuming an appropriate
rescaling), we will use the condition that $\omega\in\Event_A^1$, whence the cross-polytope is from the class $\crosspol$.
The Gaussian measure of such polytopes is computed below in Lemma~\ref{l: crosspol1},
the central statement of this section.
In the other case, when $I\cap[n]$ is ``not very large'', we will use the assumption $\omega\in\Event_A^2$,
which allows a relatively simple upper bound for the measure (essentially owing to the fact that the
expected norm of the columns of $\Gamma\Fcomp_2(A)$ is rather small).
The above description does not include the process of optimizing all the involved
parameters, which we leave for the last section.

\medskip

The next lemma is elementary; its proof is given for Reader's convenience.

\begin{lemma}
Let $P=\conv\{\pm x_1,\pm x_2,\dots,\pm x_n\}$ be a symmetric non-degenerate cross-polytope
in $\R^n$, and let $d_i:=\dist(x_i,\spn\{x_j,j<i\})$, $i\leq n$. Further, let $1\leq r\leq n$,
and let $P'=\conv\{\pm y_1,\pm y_2,\dots,\pm y_n\}$ be a symmetric cross-polytope in $\R^n$
such that $y_i=x_i$ for $i\leq r$; $\|y_i\|_2=d_i$ for $i=r+1,\dots,n$, and vectors $y_i$, $i\geq r+1$,
are mutually orthogonal and orthogonal to $\spn\{x_j,j\leq r\}$.
Then $\gamma_n(P')\geq\gamma_n(P)$.
\end{lemma}
\begin{proof}
The lemma can be proven inductively on $r$, starting with $r=n$.
Namely, we will construct a sequence of cross-polytopes $P_n,P_{n-1},\dots,P_r$, with
$P_r$ being an orthogonal transformation of $P'$ and $P_n$ an orthogonal transformation of $P$,
such that $\gamma_n(P_n)\leq \gamma_n(P_{n-1})\leq\dots\leq \gamma_n(P_r)$.

At $w$-th step ($r+1\leq w\leq n$, in this argument we step backwards),
let $P_w$ be a cross-polytope generated by $\{y_1^w,y_2^w,\dots,y_n^w\}$ where
$y_i^w$, $i=w+1,\dots,n$ are mutually orthogonal and orthogonal to $\spn\{y_j^w,\;j\leq w\}$.
Denote by $z_w$ the orthogonal projection of $y^w_w$ onto the linear span of $\spn\{y_j^w,\;j< w\}$,
and let $P_w^s$ be the symmetric cross-polytope generated by vectors
$\{y_j^w,j\neq w\}\cup\{z_w+(z_w-y^w_w)\}$.
Observe that $P_w^s$ is an orthogonal transformation of $P_w$
(reflection with respect to hyperplane orthogonal to $z_w-y^w_w$).
Now, let $P_{w-1}$ be generated by vectors $\{y_j^w,j\neq w\}\cup\{z_w-y^w_w\}$.
It is not difficult to see that $P_{w-1}$ can be represented as
$$P_{w-1}:=\Big\{\frac{1}{2}(v_1+v_2):\;v_1\in P_w,\;v_2\in P_w^s,\;\langle v_1,z_w-y_w^w\rangle
=\langle v_2,z_w-y_w^w\rangle\Big\}.$$
This, together with log-concavity of the Gaussian distribution, implies that $\gamma_n(P_{w-1})
\geq\gamma_n(P_w)$, and the result follows.
\end{proof}

As a corollary of the last lemma, we obtain
\begin{lemma}\label{l: gaussian measure simple}
Let $P=\conv\{\pm x_1,\pm x_2,\dots,\pm x_n\}$ be a symmetric cross-polytope in $\R^n$ such that
for some $1\leq r< n$ and some $0<h$ we have
$$\dist(x_i,\spn\{x_j,j<i\})\leq h,\quad\quad i\geq r+1.$$
Then $\gamma_n(P)\leq \big(\frac{eh}{n-r}\big)^{n-r}$.
\end{lemma}
\begin{proof}
In view of the previous lemma, we can assume without loss of generality that $x_i=h e_i$, $i\geq r+1$,
and that the linear span of $\{x_j,\;j\leq r\}$ coincides with $\spn\{e_1,\dots,e_r\}$.
Let $\widetilde G=(\widetilde g_1,\dots,\widetilde g_n)$ be the standard Gaussian vector in $\R^n$. Clearly,
the event $\widetilde G\in P$ implies that
$$\sum_{i=r+1}^n |\widetilde g_i|\leq h.$$
Thus,
$$\Prob\{\widetilde G\in P\}\leq (2\pi)^{-(n-r)/2}\frac{(2h)^{n-r}}{(n-r)!}
\leq \frac{(2\pi)^{-(n-r)/2-1/2}(2h)^{n-r}}{(n-r)^{n-r+1/2}e^{r-n}}\leq \bigg(\frac{eh}{n-r}\bigg)^{n-r}.$$
The result follows.
\end{proof}
As an immediate consequence of the last lemma, we get
\begin{lemma}\label{l: crosspol2}
Let $P=\conv\{\pm x_1,\pm x_2,\dots,\pm x_n\}$ be a symmetric cross-polytope
such that for some $h>0$, $\delta\in(0,1/2]$ and $1\leq k\leq n$ we have
$$\big|\big\{i:\;n-k+1\leq i\leq n,\;\dist(x_{i},\spn\{x_{j},\;j<i\})
\leq h\big\}\big|\geq (1-\delta)k.$$
Then
$\gamma_n(P)\leq \big(\frac{2eh}{k}\big)^{(1-\delta)k}$.
\end{lemma}

Estimates analogous to Lemma~\ref{l: crosspol2} are easily 
available for cross-polytopes from $\crosspol(k,h)$.
Indeed, appropriately rearranging the generating vectors, we obtain that for any cross-polytope $P$
from $\crosspol(k,h)$ we have $\gamma_n(P)\leq e^{-c'k}$ as long as $h\leq c'k$, for some universal constant $c'>0$.
However, these estimates (with such a condition on $h$)
are absolutely useless in our context: they are too weak because they do not take into consideration
invariance of $\crosspol(k,h)$ under permutations of the generating vectors.
Applying a more elaborate argument, we can derive a much stronger statement,
which is the core of our approach.

\begin{lemma}[Gaussian measure of tilted cross-polytopes from $\crosspol(k,h)$]\label{l: crosspol1}
Let
$$P=\conv\{\pm x_1,\dots,\pm x_n\}\in\crosspol(k,h),$$
with $1\leq k\leq n$ and $h\leq c_{\text{\tiny\ref{l: crosspol1}}}n$. Then $\gamma_n(P)\leq 2e^{-c_{\text{\tiny\ref{l: crosspol1}}}k}$,
for a universal constant $c_{\text{\tiny\ref{l: crosspol1}}}>0$.
\end{lemma}
\begin{proof}
Let $\widetilde G$ be the standard Gaussian vector in $\R^n$,
and let $\widetilde \Event$ be the event that $\widetilde G\in P$.
Denote
$q:=\gamma_n(P)=\Prob(\widetilde \Event)$.
Conditioned on $\widetilde \Event$, there is a (random) vector $v=(v_1,\dots,v_n)\in\R^n$ with $\|v\|_1\leq 1$
such that
$$\widetilde G=\sum\limits_{i=1}^n v_i x_i.$$
As $\Exp\big(\sum\limits_{i=1}^n |v_i|\;|\;\widetilde \Event\big)\leq 1$,
there is a non-random subset $I\subset[n]$ of cardinality $k$
such that $\Exp\big(\sum\limits_{i\in I} |v_i|\;|\;\widetilde \Event\big)\leq k/n$,
whence, by Markov's inequality, there is an event $\Event'\subset\widetilde\Event$
of probability at least $q/2$ such that
$$\sum\limits_{i\in I} |v_i|\leq\frac{2k}{n}$$
everywhere on $\Event'$.
Without loss of generality (since the definition of the class $\crosspol$ is invariant under permutations
of generating vectors), we can assume that $I=\{n-k+1,\dots,n\}$.
Now, for any $\omega\in\Event'$ we have
$$\widetilde G(\omega)=\sum\limits_{i=1}^n v_i(\omega) x_i=4\sum\limits_{i=1}^{n-k} \frac{v_i(\omega)}{4} x_i
+4\sum\limits_{i=n-k+1}^{n} \frac{nv_i(\omega)}{4k} \Big(\frac{k}{n}x_i\Big),$$
where, setting $v_i':=\frac{v_i}{4}$ for $i\leq n-k$
and $v_i':=\frac{nv_i}{4k}$ for $i>n-k$, we get $\|(v_1',\dots,v_n')\|_1\leq 1$.
Thus, if we define non-random cross-polytope $P'$ as
$$P':=4\,\conv\big\{\pm x_1,\dots,\pm x_{n-k},\pm kx_{n-k+1}/n,\dots,\pm kx_n/n\big\},$$
then
$$\Prob\{\widetilde G\in P'\}\geq q/2.$$
On the other hand, appropriately rearranging the generating vectors and using the definition of the class
$\crosspol(k,h)$, we get that there is a presentation
$$P':=\conv\{\pm y_1,\dots,\pm y_n\}$$
where for the last $\lceil k/4\rceil$ vectors $y_i$ we have
$$\dist(y_i,\spn\{y_j,j<i\})\leq \frac{4k h}{n},\quad\quad i=n-\lceil k/4\rceil+1,\dots,n.$$
Finally, applying Lemma~\ref{l: gaussian measure simple} to $P'$, we obtain
$$q/2\leq e^{-c'k},$$
for a universal constant $c'>0$.
\end{proof}

\section{Completion of the proof}\label{s: completion}

We assume that $n$ is a large positive integer, and set
$\varepsilon:=n^{-3}$; $m:=n^3$.
Although parameter $\rho$ will be determined at the very last stage, we assume that $\rho\leq n$,
whence $\varepsilon \rho n^2\leq 1$. Thus, applying Lemma~\ref{l: basic discretization},
we get that, in order to prove \eqref{eq: basis},
it is sufficient to show that the event $\Event_{\text{\tiny\ref{l: basic discretization}}}$ from Lemma~\ref{l: basic discretization}
has probability strictly less than $1-2^{-n}$.

Let $\delta:=\log^{-1} n$ and let parameters $s,\widetilde s,\tau,\alpha$ satisfy
\begin{equation}\label{eq: condit}
n\geq s\geq 4\widetilde s\geq 4\log^2 n;\quad\quad \frac{\widetilde s^2\alpha}{n}\geq C_{\text{\tiny\ref{l: parameters II}}}\log n;\quad\quad
\min\bigg(\frac{\tau^2 \widetilde s^2\alpha}{n},\frac{\tau^2 \widetilde s}{n}\bigg)\geq C_{\text{\tiny\ref{l: parameters I}}}\log^2 n.
\end{equation}
Let $\Event_A=\Event_A^1(s,\widetilde s)\cap\Event_A^2(\widetilde s,\delta,\tau)$ ($A\in\Net$) and $\Event_\cap:=\bigcap\limits_{A\in\Net}\Event_A$
be defined as in Section~\ref{s: event}.

Lemmas~\ref{l: parameters I} and~\ref{l: parameters II}, together with assumptions \eqref{eq: condit}, imply that
$\Prob(\Event_\cap)\geq 1-\frac{2}{n}$. Hence, applying Lemma~\ref{l: decomposition},
we get
\begin{align*}
\Prob(\Event_{\text{\tiny\ref{l: basic discretization}}})\leq
|\Net|\max\limits_{A\in\Net}\sup\limits_{\omega\in\Event_A}
\Prob\big(\big\{&\omega'\in\Omega:\;\mbox{For some $I\subset[2n]$ with $|I|=n$ we have}\\
&\mbox{$\widetilde G(\omega')\in 4\rho\,\big(\conv\big\{\pm\col_i(\Gamma(\omega)\Fcomp(A)),\;\;i\in I\big\}\big)$}\;\big\}\big)^{m-n^2}
+\frac{2}{n}.
\end{align*}
It is not difficult to see that, with our choice of $m,\varepsilon$, the cardinality of the set $\Net$ of matrices
can be estimated as
$$|\Net|\leq \bigg({m\choose n}(\varepsilon/3)^{-n}\bigg)^n\leq e^{Cn^2\log n}$$
for an appropriate universal constant $C>0$.
Thus, if we show that (for a specific choice of parameters) for every matrix $A\in\Net$ and $\omega\in\Event_A$ we have
\begin{align}
\Prob\big(\big\{&\omega'\in\Omega:\;\mbox{For some $I\subset[2n]$ with $|I|=n$ we have}\nonumber\\
&\mbox{$\widetilde G(\omega')\in 4\rho\,\big(\conv\big\{\pm\col_i(\Gamma(\omega)\Fcomp(A)),\;\;i\in I\big\}\big)$}\;\big\}\big)\leq 1/2,
\label{eq: indivprob}
\end{align}
this would imply that $\Prob(\Event_{\text{\tiny\ref{l: basic discretization}}})$ is close to zero, and we will obtain \eqref{eq: basis}.
For the rest of the proof, we are concerned with finding values for parameters so that
condition \eqref{eq: indivprob} is satisfied.

\bigskip

The definition of the event $\Event_A$ implies that for every $\omega\in\Event_A$
and $I\subset[2n]$, the cross-polytope $T:=\conv\big\{\pm\col_i(\Gamma(\omega)\Fcomp(A)),\;\;i\in I\big\}$
satisfies one of the three conditions:
\begin{itemize}

\item $T$ is degenerate, whence $\Prob\{\omega'\in\Omega:\;\widetilde G(\omega')\in 4\rho\,T\}=0$;

\item $T\in\crosspol(s,C_{\text{\tiny\ref{l: dist to spans I}}}\sqrt{2s})$ (if $|I\cap[n]|\geq n-\widetilde s$);

\item $|I\cap[n]|< n-\widetilde s$ and
\begin{align*}
&\big|\big\{i\in I\setminus[n]:\,\dist(\col_i(\Gamma(\omega) \Fcomp(A)),\spn\{\col_j(\Gamma(\omega) \Fcomp(A)),\;j\in I\cap[i-1]\})
\\
&\leq \tau\sqrt{\alpha|I\setminus[n]|}\big\}\big|\geq (1-\delta)|I\setminus[n]|.
\end{align*}
\end{itemize}

Observe that for each $p=0,1,\dots,n$, there are ${n\choose p}^2$ ways to choose
a subset $I\subset[2n]$ of cardinality $n$ and with $|I\cap[n]|=p$.
Thus, in order to satisfy \eqref{eq: indivprob}, it is sufficient to have
\begin{align*}
&\sum\limits_{p=0}^{n-\widetilde s-1}
{n\choose p}^2\sup\big\{\Prob\big\{\omega':\;\widetilde G(\omega')\in 4\rho\,P\big\}\;,\;P\in \mathcal U\big\}\\
+&\sum\limits_{p=n-\widetilde s}^{n}{n\choose p}^2
\sup\big\{\Prob\big\{\omega':\;\widetilde G(\omega')\in 4\rho\,P\big\}\;,\;P\in \crosspol(s,C_{\text{\tiny\ref{l: dist to spans I}}}\sqrt{2s})\big\}\leq \frac{1}{2}.
\end{align*}
Here, $\mathcal U$ denotes the collection of all cross-polytopes $P=\conv\{\pm x_1,\pm x_2,\dots,\pm x_n\}$
such that $\big|\big\{i:\;p+1\leq i\leq n,\;\dist(x_{i},\spn\{x_{j},\;j<i\})
\leq \tau\sqrt{\alpha(n-p)}\big\}\big|\geq (1-\delta)(n-p)$.
For any $p\leq n-\widetilde s-1$ and $P\in \mathcal U$,
we get, by Lemma~\ref{l: crosspol2}, that
$$\Prob\big\{\widetilde G\in 4\rho\,P\big\}\leq \bigg(\frac{8e\rho \tau\sqrt{\alpha(n-p)}}{n-p}\bigg)^{(1-\delta)(n-p)}.$$
Further, for $p\geq n-\widetilde s$ and $P\in \crosspol(s,C_{\text{\tiny\ref{l: dist to spans I}}}\sqrt{2s})$
we obtain, by Lemma~\ref{l: crosspol1},
$$\Prob\big\{\widetilde G\in 4\rho\,P\big\}\leq 2e^{-c_{\text{\tiny\ref{l: crosspol1}}}s},$$
{\it provided that} $4C_{\text{\tiny\ref{l: dist to spans I}}}\rho\sqrt{2s}\leq c_{\text{\tiny\ref{l: crosspol1}}}n$.

Summarizing, condition \eqref{eq: indivprob} is satisfied whenever we have a set of parameters $s,\widetilde s,\tau,\alpha,\rho$
satisfying conditons \eqref{eq: condit}, together with condition $4C_{\text{\tiny\ref{l: dist to spans I}}}\rho\sqrt{2s}\leq c_{\text{\tiny\ref{l: crosspol1}}}n$ and
condition
$$\sum\limits_{p=0}^{n-\widetilde s-1}
{n\choose p}^2\bigg(\frac{8e\rho \tau\sqrt{\alpha(n-p)}}{n-p}\bigg)^{(1-\delta)(n-p)}
+2\sum\limits_{p=n-\widetilde s}^{n}{n\choose p}^2
e^{-c_{\text{\tiny\ref{l: crosspol1}}}s}\leq \frac{1}{2}.$$
It is not difficult to see that 
$$\sum\limits_{p=n-\widetilde s}^{n}{n\choose p}^2
e^{-c_{\text{\tiny\ref{l: crosspol1}}}s}\leq (\widetilde s+1)\bigg(\frac{en}{\widetilde s}\bigg)^{2\widetilde s}e^{-c_{\text{\tiny\ref{l: crosspol1}}}s}
\leq \frac{1}{n},$$
as long as $s\geq C\widetilde s\log n$ for a sufficiently large universal constant $C>0$.
Further,
$$\sum\limits_{p=0}^{n-\widetilde s-1}
{n\choose p}^2\bigg(\frac{8e\rho \tau\sqrt{\alpha(n-p)}}{n-p}\bigg)^{(1-\delta)(n-p)}\leq \frac{1}{n}$$
as long as $\frac{n^2\rho\tau\sqrt{\alpha}}{\widetilde s^{5/2}}\leq c$
for a sufficiently small universal constant $c>0$.
Thus, condition \eqref{eq: indivprob} is satisfied whenever our set of parameters satisfies the relations
\begin{align*}
&n\geq s\geq C\widetilde s\log n\geq 4C\log^3 n;\quad\quad
4C_{\text{\tiny\ref{l: dist to spans I}}}\rho\sqrt{2s}\leq c_{\text{\tiny\ref{l: crosspol1}}}n;
\quad\quad \frac{\widetilde s^2\alpha}{n}\geq C_{\text{\tiny\ref{l: parameters II}}}\log n;\\
&\min\bigg(\frac{\tau^2 \widetilde s^2\alpha}{n},\frac{\tau^2 \widetilde s}{n}\bigg)\geq C_{\text{\tiny\ref{l: parameters I}}}\log^2 n;
\quad\quad\frac{n^2\rho\tau\sqrt{\alpha}}{\widetilde s^{5/2}}\leq c.
\end{align*}
Our goal is to find the largest possible $\rho$ so that the above conditions can be satisfied.
Setting
$$\tau:=\sqrt{C_{\text{\tiny\ref{l: parameters I}}}}\log n\max\big(\sqrt{n/\widetilde s},\sqrt{n/(\widetilde s^2\,\alpha)}\big),$$
we get that we can take
$$\rho:=c'\min\Big(n/\sqrt{s},\frac{1}{\log n}\frac{\widetilde s^3}{n^{5/2}\sqrt{\alpha}},
\frac{1}{\log n}\frac{\widetilde s^{7/2}}{n^{5/2}}\Big)$$
for a small enough universal constant $c'>0$.
Since it is better for us to take $\alpha$ as small as possible, we can take $\alpha:=\frac{C'n\log n}{\widetilde s^2}$
for a large enough universal constant $C'>0$.
Plugging in and optimizing over $\widetilde s$, $s$ (up to logarithmic multiples, we should take $s$, $\widetilde s$
of order $n^{8/9}$), we get the result.

\bigskip

\noindent
{\bf Acknowledgement.}
I would like to thank Assaf Naor and Pierre Youssef for introducing me to
recent results on restricted invertibility of matrices, in connection with
the Banach--Mazur distance to the cube,
Assaf Naor for suggesting to work on the problem,
Alexander Litvak and Nicole Tomczak-Jaegermann for valuable remarks.

\end{document}